\documentclass{amsart}

\usepackage[T1]{fontenc}
\usepackage[latin1]{inputenc}
\usepackage{amsmath}
\usepackage{amssymb}
\usepackage{amsthm}
\usepackage{dsfont}

\theoremstyle{plain}\newtheorem{theo}{Theorem}
\theoremstyle{plain}\newtheorem{cor}{Corollary}
\theoremstyle{definition}\newtheorem{rem}{Remark}
\theoremstyle{plain}\newtheorem{defi}{Definition}[section]
\theoremstyle{plain}\newtheorem{lem}[defi]{Lemma}
\theoremstyle{plain}\newtheorem{prop}[defi]{Proposition}
\theoremstyle{definition}\newtheorem{exa}[defi]{Example}

\newcommand{\N}{{\mathds N}}
\newcommand{\R}{{\mathds R}}
\newcommand{\Z}{{\mathds Z}}
\DeclareMathOperator{\var}{Var}
\DeclareMathOperator{\cov}{Cov}

\begin{document}

\title[Bahadur Representation for $U$-Quantiles]{Bahadur Representation for $U$-Quantiles of Dependent Data}

\author[M. Wendler]{Martin Wendler\\ \normalfont\slshape Ruhr-Universit\"at Bochum}\thanks{Research supported by the {\em Studienstiftung des deutschen Volkes} and the DFG Sonderforschungsbereich 823
{\em Statistik nichtlinearer dynamischer Prozesse}.}

\email{Martin.Wendler@rub.de}



\begin{abstract}
$U$-quantiles are applied in robust statistics, like the Hodges-Lehmann estimator of location for example. They have been analyzed in the case of independent random variables with the help of a generalized Bahadur representation. Our main aim is to extend these results to $U$-quantiles of strongly mixing random variables and functionals of absolutely regular sequences. We obtain the central limit theorem and the law of the iterated logarithm for $U$-quantiles as straightforward corollaries. Furthermore, we improve the existing result for sample quantiles of mixing data.
\end{abstract}

\maketitle

\section{Introduction}

\subsection{Sample Quantiles}

The Hodges-Lehmann estimator is defined as $H_n=\operatorname{median}\left\{\frac{X_i+X_j}{2}\big|  1\leq i<j\leq n\right\}$ and is an example of a $U$-quantile, i.e. a quantile of the sample $\left(h\left(X_i,X_j\right)\right)_{1\leq i<j\leq n}$, where $h$ is a measurable and symmetric function. $U$-statistics are decomposed into a linear part and a so-called degenerate part, so that the theory for partial sums can be applied to the linear part. Similarly, we first improve the existing results for sample quantiles.  In a second step, we use this to investigate $U$-quantiles.

This article is organized as follows: In the introduction, the definitions and some examples are given, the subsequent section contains the main results. In the third section, some preliminary results are stated and proved, the proofs of the main theorems follow in the last section. Each section is divided into a part about sample quantiles and a part about $U$-quantiles.

Let $\left(X_n\right)_{n\in\N}$ be a stationary sequence of real-valued random variables with distribution function $F$ and $p\in\left(0,1\right)$. Then the $p$-quantile $t_p$ of $F$ is defined as
\begin{equation*}
t_p=F^{-1}\left(p\right):=\inf\left\{t\in\R\big|F\left(t\right)\geq t\right\}
\end{equation*}
and can be estimated by the empirical $p$-quantile, i.e. the $\lceil \frac{n}{p} \rceil$-th order statistic of the sample $X_1\ldots,X_n$. This also can be expressed as the $p$-quantile $F^{-1}_n\left(p\right)$ of the empirical distribution function $F_n\left(t\right):=\frac{1}{n}\sum_{i=1}^{n}\mathds{1}_{X_i\leq t}$. It is clear that $F^{-1}_n\left(p\right)$ is greater than $t_p$ iff $F_n\left(t_p\right)$ is smaller than $p$. In the case of independent random variables, this converse behaviour was exploited by  Bahadur \cite{baha}, who established the representation
\begin{equation}\label{line1}
F^{-1}_n\left(p\right)=t_p+\frac{p-F_n\left(t_p\right)}{f\left(t_p\right)}+R_n
\end{equation}
(where $f=F'$ is the derivative of the distribution function) and showed that $R_n=O\left(n^{-\frac{3}{4}}(\log n)^{\frac{1}{2}}(\log\log n)^{\frac{1}{4}}\right)$. This was refined by Kiefer \cite{kief} to
\begin{equation*}
 \limsup_{n\rightarrow\infty}\left(\frac{n}{2\log\log n}\right)^{\frac{3}{4}}R_n=2^{\frac{1}{2}}3^{-\frac{3}{4}}p^{\frac{1}{4}}(1-p)^{\frac{1}{4}}.
\end{equation*}

The following short calculation shows that $R_n$ is related to the (local) empirical process $\left(F_n\left(t+t_p\right)-F_n\left(t_p\right)-f\left(t_p\right)t\right)_{t}$ centered in $\left(t_p,F_n\left(t_p\right)\right)$ and it's inverse denoted by $Z_n$:
\begin{align*}
 Z_n\left(x\right)&:=\left(F_n\left(\cdot+t_p\right)-F_n\left(t_p\right)\right)^{-1}\left(x\right)-\frac{x}{f\left(t_p\right)}\\
&=\inf\left\{s\big| F_n\left(s+t_p\right)-F_n\left(t_p\right)\leq x\right\}-\frac{x}{f\left(t_p\right)}\displaybreak[0]\\
&=\inf\left\{s\big|F_n\left(s\right)\leq x+F_n\left(t_p\right)\right\}-\frac{x}{f\left(t_p\right)}-t_p\\
&=F_n^{-1}\left(x+F_n\left(t_p\right)\right)-\frac{x}{f\left(t_p\right)}-t_p
\end{align*}
So we have
\begin{equation}
 Z_n\left(p-F_n\left(t_p\right)\right)=F_n^{-1}\left(p\right)-t_p+\frac{F_n\left(t_p\right)-p}{f\left(t_p\right)}=R_n.
\end{equation}

So the first step of our proof is showing that $\left(F_n\left(t+t_p\right)-F_n\left(t_p\right)-f\left(t_p\right)t\right)_{t\in I_n}$ converges to zero at some rate uniformly on intervalls $I_1\supset I_2\supset I_3\ldots$ By a theorem of Vervaat, $-Z_n$ has the same limit behaviour as the (local) empirical process. We will then conclude that $R_n=Z_n\left(F\left(t_p\right)-F_n\left(t_p\right)\right)$ converges to zero at the same rate and obtain the central limit theorem and the law of the iterated logarithm as easy corollaries.

There is a broad literature on the Bahadur representation for dependent data beginning with Sen \cite{sen}, who studied $\phi$-mixing random variables. Babu and Singh \cite{babu} proved such a representation under an exponentially fast decay of the strong mixing coefficients, this was weakened by Yoshihara \cite{yos2} and Sun \cite{sun} to a polynomial decay of the strong mixing coefficients. Hesse \cite{hess}, Wu \cite{wu} and Kulik \cite{kuli} established a Bahadur representation for linear processes. The first aim of this paper is to give better rates than Sun under polynomial strong mixing.

\begin{defi} Let $\left(X_n\right)_{n\in\N}$ be a stationary process. Then the strong mixing coefficients are defined as
\begin{equation}
 \alpha (k) := \sup \left\{\left| P[AB] - P[A]P[B] \right| : A \in \mathcal{F}^k_1, B \in \mathcal{F}^\infty_{n+k}, n \in \N \right\}
\end{equation}
where $\mathcal{F}^l_a$ is the $\sigma$-field generated by random variables $X_a, \ldots, X_l$. We say that $\left(X_n\right)_{n\in\mathds{N}}$ is strongly mixing if $\lim_{k \rightarrow \infty} \alpha (k) = 0$.
\end{defi}
For further information on strong mixing and a detailed description of the other mixing assumptions, see Bradley \cite{brad}. The assumption of strong mixing is very common, but does not cover all relevant classes of processes. For linear processes with discrete innovations or for data from dynamical systems this condition does not hold. Therefore, we will consider functionals of absolutely regular processes:

\begin{defi} Let $\left(X_n\right)_{n\in\N}$ be a stationary process. Then the absolute regularity coefficient is given by
\begin{equation}
\beta (k) = \sup_{n\in\N}E \sup \{ \left| P[A | \mathcal{F}_{-\infty}^n] - P[A]\right| : A \in \mathcal {F}^\infty_{n + k}\},
\end{equation}
and $\left(X_n\right)_{n\in\N}$ is called absolutely regular, if $\beta(k)\rightarrow0$ as $k\rightarrow\infty.$
\end{defi}

We call a sequence $\left(X_n\right)_{n\in\mathds{Z}}$ a two-sided functional of $\left(Z_n\right)_{n\in\mathds{Z}}$ if there is a measurable function defined on $\R^\Z$ such that
\begin{equation}
X_n = f\left( \left(Z_{n + k}\right)_{k \in \Z}\right).
\end{equation}
In addition we will assume that $\left(X_n\right)_{n\in\mathds{Z}}$ satisfies the 1-approximation condition:
\begin{defi} We say that $\left(X_n\right)_{n\in\mathds{Z}}$ is an $1$-approximating functional of $\left(Z_n\right)_{n\in\mathds{Z}}$, if
\begin{equation}
E \left| X_1 - E\left[X_1 \big| \mathcal {F}^l_{- l}\right] \right| \leq a_l \qquad l = 0, 1,2 \ldots
\end{equation}
where $\lim_{l \rightarrow  \infty} a_l = 0$ and $ \mathcal {F}_{-l}^l$ is the $\sigma$-field generated by $Z_{-l}, \ldots, Z_l.$
\end{defi}
This class of dependent sequences covers data from dynamical systems, which are deterministic in the sense that there exists a map $T$ such that $X_{n+1}=T\left(X_n\right)$. For example, the map $T\left(x\right)=\frac{1}{x}-\lfloor\frac{1}{x}\rfloor$ is related to the continued fraction
\begin{equation*}
X_n = f\left( \left(Z_{n + k}\right)_{k \in \N}\right)=\frac{1}{Z_n+\frac{1}{Z_{n+1}+\frac{1}{Z_{n+2}+\ldots}}}
\end{equation*}
where $\left(Z_n\right)_{n\in\N}$ is a stationary, absolutely regular process (even uniformly mixing, see Billingsley \cite{bill}, p. 50) taking values in $\N$ if the distribution of $X_0$ is the Gauss measure given by the density $f\left(x\right)=\frac{1}{\log 2}\frac{1}{1+x}$.

Linear processes (where the innovations are allowed to be discrete and dependent) are also functionals of absolutely regular processes. Let $\left(Z_n\right)_{n\in\mathds{Z}}$ be a stationary, absolutely regular process with $E\left|Z_1\right|<\infty$ and $\left(c_k\right)_{k\in\N}$ a real valued sequence with $\sum_{k=1}^{\infty}\left|c_k\right|<\infty$. Then for $X_n=\sum_{k=1}^{\infty}c_kZ_{n-k}$:
\begin{multline*}
E \left| X_1 - E\left[X_1 \big| \mathcal {F}^l_{- l}\right] \right|=E\left|\sum_{k=l+1}^{\infty}c_{k}\left(Z_{1-k}-E\left[Z_{1-k}\big|\mathcal {F}^l_{- l}\right]\right)\right|\\
\leq\sum_{k=l+1}^{\infty}\left|c_k\right|2E\left|Z_1\right|=:a_l\xrightarrow{l\rightarrow\infty}0.
\end{multline*}

The second aim of this paper is to establish a Bahadur representation for functionals of absolutely regular processes. If $\left(X_n\right)_{n\in\mathds{Z}}$ is an approximating function with constants $\left(a_l\right)_{l\in\N}$, it is not clear that the same holds for $\left(g\left(X_n\right)\right)_{n\in\N}$. We therefore need an additional continuity condition:
\begin{defi}\label{def2} Let $\left(X_n\right)_{n\in\mathds{N}}$ be a stationary process.
\begin{enumerate}
\item A function $g:\R\rightarrow\R$ satisfies the variation condition, if there is a constant $L$ such that
\begin{equation}\label{line6}
 E\left[\sup_{\left\|x-X_0\right\|\leq \epsilon,\ \left\|x'-X_0\right\|\leq \epsilon}\left|g\left(x\right)-g\left(x'\right)\right|\right]\leq L\epsilon.
\end{equation}
 \item A function $g:\R\times\R\rightarrow\R$ satisfies the uniform variation condition on $B\subset\R$, if there is a constant $L$ such that Line (\ref{line6}) holds for all functions $g\left(\cdot,t\right)$, $t\in B$.
 \end{enumerate}
\end{defi}
Obviously, every Lipschitz-continuous function satisfies this condition, but our main example are indicator functions. However, the variation condition can also hold for such discontinuous functions:
\begin{exa}
Let $g\left(x,t\right)=\mathds{1}_{\left\{x\leq t\right\}}$. Then
\begin{equation*}
\sup_{\left\|x-X_0\right\|\leq \epsilon,\ \left\|x'-X_0\right\|\leq \epsilon}\left|g\left(x,t\right)-g\left(x',t\right)\right|=\begin{cases}1 &\text{ if } X_0\in\left(t-\epsilon,t+\epsilon\right]\\0&\text{ else}\end{cases}.
\end{equation*}
Hence
\begin{equation*}
E\left[\sup_{\left\|x-X_0\right\|\leq \epsilon,\ \left\|x'-X_0\right\|\leq \epsilon}\left|g\left(x,t\right)-g\left(x',t\right)\right|\right]\leq F\left(t+\epsilon\right)-F\left(t-\epsilon\right)\leq L\epsilon
\end{equation*}
uniformly on $\R$, if $F$ is Lipschitz-continuous.
\end{exa}

\subsection{$U$-Quantiles}

U-quantiles are applied in robust estimation, for example the Hodges-Lehmann estimator of location. It has a breakdown point of 29\%, that means 29\% of the random variables can be replaced by random variables with different distribution before the estimation breaks down completely (see Huber \cite{hube} for details). It is also very efficient in the case of independent normal distributed random variables.

Let $h:\R\times\R\rightarrow\R$ be a measurable, symmetric function. We are interested in the empirical $U$-quantile, i.e. the $p$-quantile of the sample $\left(h\left(X_i,X_j\right)\right)_{1\leq i<j\leq n}$, which can be expressed by $U^{-1}_n\left(p\right)$ with $U_n\left(t\right):=\frac{2}{n(n-1)}\sum_{1\leq i<j\leq n}\mathds{1}_{h\left(X_i,X_j\right)\leq t}$. Let $U\left(t\right):=P\left[h\left(X,Y\right)\leq t\right]$ ($X$, $Y$ being independent random variables with the same distribution as $X_1$) be differentiable in $U^{-1}\left(p\right)$ with $u\left(U^{-1}\left(p\right)\right):=U'\left(U^{-1}\left(p\right)\right)>0$. Similarly to a sample quantile, $U^{-1}_n\left(p\right)$ can be analyzed with the help of a generalized Bahadur respresentation
\begin{equation}
U^{-1}_n\left(p\right)=U^{-1}\left(p\right)+\frac{U\left(U^{-1}\left(p\right)\right)-U_n\left(U^{-1}\left(p\right)\right)}{u\left(U^{-1}\left(p\right)\right)}+R'_n.
\end{equation}

For the special case of the Hodges-Lehmann estimator of independent data, Geertsema \cite{geer} established a generalized Bahadur representation with $R'_n=O\left(n^{-\frac{3}{4}}\log n\right)$ a.s.. For general $U$-quantiles, Dehling, Denker, Philipp \cite{deh3} and Choudhury and Serfling \cite{chou} improved the rate to $R'_n=O\left(n^{-\frac{3}{4}}(\log n)^{\frac{3}{4}}\right)$. Arcones \cite{arco} proved the exact order $R'_n=O\left(n^{-\frac{3}{4}}(\log\log n)^{\frac{3}{4}}\right)$ as for sample quantiles. We use a slightly more general definition:
\begin{defi}\label{def1}
We call a nonnegative, measurable function $h:\R\times\R\times\R\rightarrow\R$, which is symmetric in the first two arguments and nondecreasing in the third argument, a kernel function. For fixed $t\in\R$, we call
\begin{equation}
 U_n\left(t\right):=\frac{2}{n(n-1)}\sum_{1\leq i<j\leq n}h\left(X_i,X_j,t\right)
\end{equation}
the $U$-statistic with kernel $h\left(\cdot,\cdot,t\right)$ and the process $\left(U_n\left(t\right)\right)_{t\in\R}$ the empirical $U$-distribution function. We define the $U$-distribution function as $U\left(t\right):=E\left[h\left(X,Y,t\right)\right]$, where $X$, $Y$ are independent with the same distribution as $X_1$.

$U_n^{-1}(p)$ is called empirical $p$-$U$-quantile.
\end{defi}
In order to prove asymptotic normality, Hoeffding \cite{hoef} decomposed $U$-statistics into a linear and a so-called degenerate part:
\begin{equation}
U_n\left(t\right)=U\left(t\right)+\frac{2}{n}\sum_{i=1}^{n}h_{1}\left(X_{i},t\right)+\frac{2}{n\left(n-1\right)}\sum_{1\leq i<j\leq n}h_{2}\left(X_{i},X_{j},t\right)
\end{equation}
where
\begin{align*}
h_1(x,t)&:=Eh(x,Y,t)-U\left(t\right) \\
h_2(x,y,t)&:=h(x,y,t) - h_1(x,t) -h_1(y,t) -U\left(t\right).
\end{align*}

$U$-statistics and $U$-processes have been investigated not only for independent data, but also for different classes of dependent data: Sen \cite{sen2} considered $\star$-mixing observations, Yoshihara \cite{yosh} studied absolutely regular observations, Denker and Keller \cite{denk} functionals of absolutely regular processes. Borovkova, Burton, Dehling \cite{boro} extended this to $U$-processes. Hsing, Wu \cite{hsin} investigated $U$-statistics for some class of causal processes and Dehling, Wendler \cite{dehl}, \cite{deh2} for strongly mixing oberservations. As far as we know there are no results on $U$-quantiles of dependent data, our third and main aim is to give a rate of convergence of the remainder term in the Bahahdur-representation of $U$-quantiles for strongly mixing sequences and for functionals of absolutely regular sequences. The central limit theorem and the law of the iterated logarithm for $U$-quantiles are straightforward corollaries.

Similar to sample quantiles, we need special continuity assumptions on the kernel:
\begin{defi}\label{def3} Let $\left(X_n\right)_{n\in\mathds{N}}$ be a stationary process and $t\in\R$.
\begin{enumerate}
\item The kernel $h$ satisfies the variation condition for $t\in\R$, if there is a constant $L$ such that
\begin{equation}\label{line8}
 E\left[\sup_{\left\|(x,y)-(X,Y)\right\|\leq \epsilon,\ \left\|(x',y')-(X,Y)\right\|\leq \epsilon}\left|h\left(x,y,t\right)-h\left(x',y',t\right)\right|\right]\leq L\epsilon,
\end{equation}
where $X$, $Y$ are independent with the same distribution as $X_1$ and $\left\|(x_1,x_2)\right\|=(x_1^2+x_2^2)^{1/2}$ denotes the Euclidean norm.
 \item The kernel $h$ satisfies the uniform variation condition on $B\subset\R$, if there is a constant $L$ such that Line (\ref{line8}) holds for all $t\in B$.
 \end{enumerate}
\end{defi}

\begin{exa}[Hodges-Lehmann estimator] Let $h\left(x,y,t\right)=\mathds{1}_{\left\{\frac{1}{2}\left(x+y\right)\leq t\right\}}$. The 0.5-$U$-quantil is the Hodges-Lehmann estimator for location \cite{hodg}. Note that
\begin{equation*}
\sup_{\substack{\left\|(x,y)-(X,Y)\right\|\leq \epsilon\\ \left\|(x',y')-(X,Y)\right\|\leq \epsilon}}\left|\mathds{1}_{\left\{\frac{1}{2}\left(x+y\right)\leq t\right\}}-\mathds{1}_{\left\{\frac{1}{2}\left(x'+y'\right)\leq t\right\}}\right|=\begin{cases}1 &\text{ if } \frac{X+Y}{2}\in\left(t-\frac{\epsilon}{\sqrt{2}},t+\frac{\epsilon}{\sqrt{2}}\right]\\0&\text{ else}\end{cases}
\end{equation*}
If $X_1$ has a bounded density, then the density $f_{\frac{1}{2}\left(X+Y\right)}$ of $\frac{1}{2}\left(X+Y\right)$ is also bounded, so
\begin{multline*}
 E\left[\sup_{\left\|(x,y)-(X,Y)\right\|\leq \epsilon,\ \left\|(x',y')-(X,Y)\right\|\leq \epsilon}\left|h\left(x,y\right)-h\left(x',y'\right)\right|\right]\\
\leq P\left[\frac{X+Y}{2}\in\left(t-\frac{\epsilon}{\sqrt{2}},t+\frac{\epsilon}{\sqrt{2}}\right]\right]\leq \left(\sqrt{2}\sup_{x\in\R}f_{\frac{1}{2}\left(X+Y\right)}(x)\right)\cdot\epsilon
\end{multline*}
and $\mathds{1}_{\left\{\frac{1}{2}\left(x+y\right)\leq t\right\}}$ satisfies the uniform variation condition on $\R$.
\end{exa}

\begin{exa}[$Q_n$ estimator of scale]
Let $h\left(x,y,t\right)=\mathds{1}_{\left\{\left|x-y\right|\leq t\right\}}$. When the 0.25-$U$-quantile is the $Q_n$ estimator of scale proposed by Rousseeuw and Croux \cite{rous}. If $X_1$ has a bounded density, then with similar arguments as for the Hodges-Lehmann-estimator, $\mathds{1}_{\left\{\left|x-y\right|\leq t\right\}}$ satisfies the uniform variation condition.
\end{exa}

\section{Main results}

\subsection{Sample Quantiles}

In the following theorems we assume that $\left(X_n\right)_{n\in\N}$ is a stationary process.
\begin{theo}\label{theo1}Let $g:\R\times\R\rightarrow\R$ be a nonnegative, bounded, measurable function which is nondecreasing in the second argument, let $F\left(t\right):=E\left[g\left(X_1,t\right)\right]$ be differentiable in $t_p\in\R$ with $F'\left(t_p\right)=f\left(t_p\right)>0$ and
\begin{equation}\label{line9}
 \left|F\left(t\right)-F\left(t_p\right)-f\left(t_p\right)\left(t-t_p\right)\right|=o\left(\left|t-t_p\right|^{\frac{3}{2}}\right)\ \ \ \text{as}\ \ t\rightarrow t_p.
\end{equation}
Assume that one of the following two conditions holds:
\begin{enumerate}
\item $\left(X_n\right)_{n\in\N}$ is strongly mixing with $\alpha\left(n\right)=O\left(n^{-\beta}\right)$ for some $\beta\geq3$. Let $\gamma:=\frac{\beta-2}{\beta}$.
\item $\left(X_n\right)_{n\in\N}$ is an $1$-approximating functional of an absolutely regular process $\left(Z_n\right)_{n\in\mathds{Z}}$ with mixing coefficients $\left(\beta(n)\right)_{n\in\N}$ and approximation constants $\left(a_n\right)_{n\in\N}$, such that $\beta(n)=\left(n^{-\beta}\right)$ and $a_n=\left(n^{-(\beta+3)}\right)$ for some $\beta>3$. Let $g$ satisfy the variation condition uniformly in some neighbourhood of $t_p$ and let $\gamma:=\frac{\beta-3}{\beta+1}$.
\end{enumerate}
Then for $F_n\left(t\right):=\frac{1}{n}\sum_{i=1}^{n}g\left(X_i,t\right)$, $p=F\left(t_p\right)$ and any constant $C>0$
\begin{align}\label{line10}
 \sup_{\left|t-t_p\right|\leq C\sqrt{\frac{\log\log n}{n}}}\left|F_n\left(t\right)-F\left(t\right)-F_n\left(t_p\right)+F\left(t_p\right)\right|=o\left(n^{-\frac{5}{8}-\frac{1}{8}\gamma}(\log n)^{\frac{3}{4}}(\log\log n)^{\frac{1}{2}}\right)\\
\label{line11}R_n:=F^{-1}_n\left(p\right)-t_p+\frac{F\left(t_p\right)-F_n\left(t_p\right)}{f\left(t_p\right)}=o\left(n^{-\frac{5}{8}-\frac{1}{8}\gamma}(\log n)^{\frac{3}{4}}(\log\log n)^{\frac{1}{2}}\right)
\end{align}
a.s. as $n\rightarrow\infty$.
\end{theo}

\begin{rem}\label{rem1} Bahadur representations for sample quantiles of strongly mixing data have previously been established by Yoshihara \cite{yos2} and Sun \cite{sun}. Yoshihara states the rate $R_n=o\left(n^{-\frac{3}{4}}\log n\right)$ a.s., but a careful reading shows that there is a mistake in Line (20) of his paper, which has to be
\begin{equation*}
E\left|\sum_{j=1}^n\sum_{i=1}^l\zeta_j\left(\theta+(i-1)q_k,\theta+iq_k\right)\right|^4\leq n^2(lq_k)^{1+\gamma}.
\end{equation*}
His proof leads to our rate with $\gamma\leq\frac{1}{5}$ instead of our $\gamma=\frac{\beta-2}{\beta}\in\left[\frac{1}{3},1\right)$. Sun assumes a faster decay of the mixing coefficients, namely $\beta>10$, and obtains the rate $R_n=o\left(n^{-\frac{3}{4}+\delta}\log n\right)$ for any $\delta>\frac{11}{4(\beta+1)}$.
\end{rem}

\begin{rem}\label{rem2} Our condition in Line (\ref{line9}) is fullfilled if $F$ is twice differentiable in $t_p$. This is weaker than $F$ being twice differentiable in a neighbourhood of $t_p$ as required by Bahadur \cite{baha}, Yoshihara \cite{yos2} and Sun \cite{sun}. 
\end{rem}

\begin{cor}\label{cor1} Under the assumptions of Theorem \ref{theo1} it holds that
\begin{equation}
 \sqrt{n}\left(F_n^{-1}\left(p\right)-t_p\right)\xrightarrow{\mathcal{D}}N\left(0,\sigma^2\right)
\end{equation}
where
\begin{equation*}
\sigma^2=\frac{1}{f^2(t_p)}\left(\var\left[g\left(X_1,t_p\right)\right]+2\sum_{k=2}^\infty\cov\left[g\left(X_1,t_p\right),g\left(X_k,t_p\right)\right]\right).
\end{equation*}
Under Condition 1. a.s.
\begin{equation}
\limsup_{n\rightarrow\infty}\pm\sqrt{\frac{n}{\log\log n}}\left(F_n^{-1}\left(p\right)-t_p\right)=\sqrt{2\sigma^2}.
\end{equation}
Under Condition 2., the sequence $\sqrt{\frac{n}{\log\log n}}\left(F_n^{-1}\left(p\right)-t_p\right)$ is a.s. bounded.
\end{cor}

\begin{proof} This Corollary follows directly by the central limit theorem for $F_n\left(t_p\right)$ (Theorem 1.4 of Ibragimov \cite{ibra}, Theorem 4 of Borovkova et al. \cite{boro}) respectively the law of the iterated logarithm (Theorem 3 of Rio \cite{rio}, Proposition \ref{theo3b}), the Bahadur representation (\ref{line1}) and Line (\ref{line11}).
\end{proof}

\subsection{$U$-Quantiles}

\begin{theo}\label{theo2}Let $h:\R\times\R\times\R\rightarrow\R$ be a bounded kernel function that satisfies the uniform variation condition in some neighbourhood of $t_p$. Let $U\left(t\right):=E\left[h\left(X,Y,t\right)\right]$ be differentiable in $t_p\in\R$ with $U'\left(t_p\right)=u\left(t_p\right)>0$ and
\begin{equation}\label{line15}
 \left|U\left(t\right)-U\left(t_p\right)-u\left(t_p\right)\left(t-t_p\right)\right|=o\left(\left|t-t_p\right|^{\frac{3}{2}}\right)\ \ \ \text{as}\ \ t\rightarrow t_p.
\end{equation}
Assume that one of the following two conditions holds:
\begin{enumerate}
\item $\left\|X_n\right\|_1<\infty$ and $\left(X_n\right)_{n\in\N}$ is strongly mixing and the mixing coefficients satisfy $\alpha\left(n\right)=O\left(n^{-\beta}\right)$ for some $\beta\geq\frac{13}{4}$. Le $\gamma:=\frac{\beta-2}{\beta}$.
\item $\left(X_n\right)_{n\in\N}$ is an $1$-approximating functional of an absolutely regular process $\left(Z_n\right)_{n\in\mathds{Z}}$ with mixing coefficients $\left(\beta(n)\right)_{n\in\N}$ and approximation constants $\left(a_n\right)_{n\in\N}$, such that $\beta(n)=\left(n^{-\beta}\right)$ and $a_n=\left(n^{-(\beta+3)}\right)$ for some $\beta>3$. Let $\gamma:=\frac{\beta-3}{\beta+1}$.
\end{enumerate}
Then for $U_n\left(t\right):=\frac{2}{n(n-1)}\sum_{1\leq i<j\leq n}h\left(X_i,X_j,t\right)$, $p=U\left(t_p\right)$ and any constant $C>0$
\begin{align}
 \label{line16}\sup_{\left|t-t_p\right|\leq C\sqrt{\frac{\log\log n}{n}}}\left|U_n\left(t\right)-U\left(t\right)-U_n\left(t_p\right)+U\left(t_p\right)\right|=o\left(n^{-\frac{5}{8}-\frac{1}{8}\gamma}(\log n)^{\frac{3}{4}}(\log\log n)^{\frac{1}{2}}\right)\\
\label{line17}R'_n:=U^{-1}_n\left(p\right)-t_p+\frac{U\left(t_p\right)-U_n\left(t_p\right)}{u\left(t_p\right)}=o\left(n^{-\frac{5}{8}-\frac{1}{8}\gamma}(\log n)^{\frac{3}{4}}(\log\log n)^{\frac{1}{2}}\right)
\end{align}
a.s. as $n\rightarrow\infty$.
\end{theo}

\begin{cor}\label{cor2} Under the assumptions of Theorem \ref{theo2} it holds that
\begin{equation}
 \sqrt{n}\left(U_n^{-1}\left(p\right)-t_p\right)\xrightarrow{\mathcal{D}}N\left(0,\sigma^2\right)
\end{equation}
with
\begin{equation*}
\sigma^2=\frac{1}{f^2(t_p)}\left(\var\left[h_1\left(X_1,t_p\right)\right]+2\sum_{k=2}^\infty\cov\left[h_1\left(X_1,t_p\right),h_1\left(X_k,t_p\right)\right]\right).
\end{equation*}
Under Condition 1. a.s.
\begin{equation}
\limsup_{n\rightarrow\infty}\pm\sqrt{\frac{n}{\log\log n}}\left(U_n^{-1}\left(p\right)-t_p\right)=\sqrt{2\sigma^2}.
\end{equation}
Under Condition 2., the sequence $\sqrt{\frac{n}{\log\log n}}\left(U_n^{-1}\left(p\right)-t_p\right)$ is bounded a.s.
\end{cor}

\begin{proof} This Corollary is an easy consequence of Line (\ref{line17}) and Proposition \ref{theo5} respectively Proposition \ref{theo4} or \ref{theo4b}.

\end{proof}

\section{Preliminary results}

\subsection{Sample Quantiles}

In this section, we recall some existing lemmas for handy reference and prove some technical results. In the proofs, $C$ denotes an arbitrary constant, which may have different values from line to line and may depend on several other values, but not on $n\in\N$. An important tool in the analysis of weakly dependent random variables are covariance inequalities:

\begin{lem}[Davydov \cite{davy}]\label{lem1} If $Y_1$ and $Y_2$ are random variables such that $Y_1$ is measurable with resprect to $\mathcal{F}_1^k$ and $Y_2$ with respect to $\mathcal{F}_{k+n}^\infty$ for some $k\in\N$, then
\begin{equation*}
 \left|E\left[Y_1Y_2\right]-E\left[Y_1\right]E\left[Y_2\right]\right|\leq 10\left\|Y_1\right\|_{p_1}\left\|Y_2\right\|_{p_2}\alpha^{\frac{1}{p_3}}\left(n\right)
\end{equation*}
for all $p_1,p_2,p_3\in\left[0,1\right]$ with $\frac{1}{p_1}+\frac{1}{p_2}+\frac{1}{p_3}=1$.
\end{lem}

\begin{lem}[Borovkova et al. \cite{boro}]\label{lem2} Let $\left(X_n\right)_{n\in\N}$ be an $1$-approximating functional with approximation constants $\left(a_l\right)_{l\in\N}$ of an absolutely regular process $\left(Z_n\right)_{n\in\N}$ and $\left\|X_0\right\|_{2+\delta}<\infty$ for some $\delta>0$. Then
\begin{equation*}
 \left|E\left[X_iX_{i+k}\right]-\left(EX_i\right)\left(EX_k\right)\right|\leq2\left\|X_0\right\|_{2+\delta}^{2}\left(\beta\left(\lfloor \frac{k}{3}\rfloor\right)\right)^{\frac{\delta}{2+\delta}}+4\left\|X_0\right\|_{2+\delta}^{\frac{2+\delta}{1+\delta}}a_{\lfloor \frac{k}{3}\rfloor}^{\frac{\delta}{1+\delta}}.
\end{equation*}
\end{lem}

\begin{lem}[Borovkova et al. \cite{boro}]\label{lem3} Let $\left(X_n\right)_{n\in\N}$ be a bounded $1$-approximation functional with approximation constants $\left(a_l\right)_{l\in\N}$ of an absolutely regular process $\left(Z_n\right)_{n\in\N}$. Then
\begin{multline*}
 \left|E\left[X_iX_jX_kX_l\right]-E\left[X_i\right]E\left[X_jX_kX_l\right]\right|\\
\leq\left(6\left\|X_0\right\|_{2+\delta}^{2}\left(\beta\left(\lfloor \frac{j-i}{3}\rfloor\right)\right)^{\frac{\delta}{2+\delta}}+8\left\|X_0\right\|_{2+\delta}^{\frac{2+\delta}{1+\delta}}a_{\lfloor \frac{j-i}{3}\rfloor}^{\frac{\delta}{1+\delta}}\right)\left\|X_0\right\|_{\infty}^{2}
\end{multline*}
and
\begin{multline*}
 \left|E\left[X_iX_jX_kX_l\right]-E\left[X_iX_j\right]E\left[X_kX_l\right]\right|\\
\leq\left(6\left\|X_0\right\|_{2+\delta}^{2}\left(\beta\left(\lfloor \frac{k-j}{3}\rfloor\right)\right)^{\frac{\delta}{2+\delta}}+8\left\|X_0\right\|_{2+\delta}^{\frac{2+\delta}{1+\delta}}a_{\lfloor \frac{k-j}{3}\rfloor}^{\frac{\delta}{1+\delta}}\right)\left\|X_0\right\|_{\infty}^{2}.
\end{multline*}
\end{lem}

In the analysis of empirical processes, fourth moment inequalities are often used:

\begin{lem}\label{lem4} Let $\left(X_n\right)_{n\in\N}$ be a stationary, strongly mixing sequence with $\alpha\left(n\right)=O\left(n^{-\beta}\right)$ for some $\beta>3$ and $C_1,C_2>0$ constants. Then there exists a constant $C$, such that for all measurable, nonnegative functions $g:\R\rightarrow\R$ bounded by $C_1$ and with $E\left|g\left(X_1\right)-Eg\left(X_1\right)\right|\geq C_2n^{-\frac{\beta}{\beta+1}}$ and all $n\in\N$
\begin{equation*}
 E\left(\sum_{i=1}^{n}g\left(X_i\right)-E\left[g\left(X_1\right)\right]\right)^4\leq C n^2\left(\log n\right)^2\left(E\left|g\left(X_1\right)\right|\right)^{1+\gamma}
\end{equation*}
with $\gamma=\frac{\beta-2}{\beta}$.
\end{lem}

\begin{proof}
We define the random variables $Y_i=g\left(X_i\right)-Eg\left(X_1\right)$. Using Lemma \ref{lem1} with $p_1=p_2=\frac{2\beta}{\beta-3}$ and $p_3=\frac{\beta}{3}$ we obtain the following three inequalities for all $i,j,k\in\N$:
\begin{align*}
 \left|E\left[Y_0Y_iY_{i+j}Y_{i+j+k}\right]\right|&\leq C\alpha^{\frac{3}{\beta}}\left(i\right)\left\|Y_0\right\|_{\frac{2\beta}{\beta-3}}\left\|Y_0Y_jY_{j+k}\right\|_{\frac{2\beta}{\beta-3}},\\
 \left|E\left[Y_0Y_iY_{i+j}Y_{i+j+k}\right]\right|&\leq C\left|E\left[Y_0Y_i\right]\right|\left|E\left[Y_0Y_k\right]\right|+ C\alpha^{\frac{3}{\beta}}\left(j\right)\left\|Y_0Y_i\right\|_{\frac{2\beta}{\beta-3}}\left\|Y_0Y_k\right\|_{\frac{2\beta}{\beta-3}},\\
 \left|E\left[Y_0Y_iY_{i+j}Y_{i+j+k}\right]\right|&\leq C\alpha^{\frac{3}{\beta}}\left(k\right)\left\|Y_0Y_iY_{i+j}\right\|_{\frac{2\beta}{\beta-3}}\left\|Y_0\right\|_{\frac{2\beta}{\beta-3}}.\\
\end{align*}
By the same lemma with $p_1=p_2=\frac{2\beta}{\beta-1}$ and $p_3=\beta$, we get
\begin{equation*}
 \left|E\left[Y_0Y_i\right]\right|\leq C\alpha^{\frac{1}{\beta}}\left(i\right)\left\|Y_1\right\|^2_{\frac{2\beta}{\beta-1}}.
\end{equation*}
As $Y_n$ is bounded, we have that $\left\|Y_0\right\|_{\frac{2\beta}{\beta-3}}\leq C\left(E\left|Y_1\right|\right)^{\frac{\beta-3}{2\beta}}$, $\left\|Y_0Y_jY_{j+k}\right\|_{\frac{2\beta}{\beta-3}}\leq C\left(E\left|Y_1\right|\right)^{\frac{\beta-3}{2\beta}}$, $\left\|Y_1\right\|_{\frac{2\beta}{\beta-1}}\leq C\left(E\left|Y_1\right|\right)^{\frac{\beta-1}{2\beta}}$ and it follows that
\begin{equation*}
 \left|E\left[Y_0Y_iY_{i+j}Y_{i+j+k}\right]\right|\leq C\alpha^{\frac{1}{\beta}}\left(i\right)\alpha^{\frac{1}{\beta}}\left(k\right)\left(E\left|Y_1\right|\right)^{\frac{2\beta-2}{\beta}}+ C\alpha^{\frac{3}{\beta}}\left(\max\left\{i,j,k\right\}\right)\left(E\left|Y_1\right|\right)^{\frac{\beta-3}{\beta}}.
\end{equation*}
Now by stationarity it is
\begin{multline*}
 E\left(\sum_{i=1}^{n}Y_i\right)^4\leq Cn\sum_{i,j,k=1}^{n}\left|E\left[Y_0Y_iY_{i+k}Y_{i+k+j}\right]\right|\\
\leq Cn^2\sum_{i=1}^{n}\alpha^{\frac{1}{\beta}}\left(i\right)\sum_{k=1}^{n}\alpha^{\frac{1}{\beta}}\left(k\right)\left(E\left|Y_1\right|\right)^{\frac{2\beta-2}{\beta}}
+Cn\sum_{i=1}^{n}i^2\alpha^{\frac{3}{\beta}}\left(i\right)\left(E\left|Y_1\right|\right)^{\frac{\beta-3}{\beta}}.
\end{multline*}
As $E\left|g\left(X_1\right)\right|\geq C_2n^{-\frac{\beta}{\beta+1}}$, we have that $\left(E\left|Y_1\right|\right)^{\frac{\beta-3}{\beta}}\leq Cn\left(E\left|Y_1\right|\right)^{\frac{2\beta-2}{\beta}}$ and with $\alpha\left(n\right)=O\left(n^{-\beta}\right)$, we arrive at
\begin{multline*}
E\left(\sum_{i=1}^{n}Y_i\right)^4\leq Cn^2\sum_{i=1}^{n}\frac{1}{i}\sum_{k=1}^{n}\frac{1}{k}\left(E\left|Y_1\right|\right)^{\frac{2\beta-2}{\beta}}
+Cn^2\sum_{i=1}^{n}i^2\frac{1}{i^3}\left(E\left|Y_1\right|\right)^{\frac{2\beta-2}{\beta}}\\
\leq C n^2\left(\log n\right)^2\left(E\left|Y_1\right|\right)^{\frac{2\beta-2}{\beta}}=C n^2\left(\log n\right)^2\left(E\left|Y_1\right|\right)^{1+\gamma}.
\end{multline*}
\end{proof}

If $(X_n)_{n\in\N}$ is an $1$-approximating functional and $g$ an arbitrary function, it is not clear that the same holds for $(g(X_n))_{n\in\N}$, so we give the following lemma:

\begin{lem}\label{lem6} Let $\left(X_n\right)_{n\in\N}$ be an $1$-approximating functional of an absolutely regular process $\left(Z_n\right)_{n\in\Z}$ with approximation constants $\left(a_n\right)_{n\in\N}$ and let $g$ be a function bounded by $K$ and satisfy the variation condition with constant $L$. Then $\left(g\left(X_n\right)\right)_{n\in\N}$ is an $1$-approximating functional with approximation constants $\left((L+K)\sqrt{a_n}\right)_{n\in\N}$. 
\end{lem}

\begin{proof}
By the Markov inequality we have that
\begin{equation*}
 P\left[\left|X_0-E[X_0\big|\mathcal{F}_{-l}^{l}]\right|\geq \sqrt{a_l}\right]\leq \frac{E\left|X_0-E[X_0\big|\mathcal{F}_{-l}^{l}]\right|}{\sqrt{a_l}}\leq \sqrt{a_l}.
\end{equation*}
We conclude that
\begin{align*}
&E\left[g\left(X_0\right)-g\left(E[X_0\big|\mathcal{F}_{-l}^{l}]\right)\right]\\
=&E\left[\left(g\left(X_0\right)-g\left(E[X_0|\mathcal{F}_{-l}^{l}]\right)\right)\mathds{1}_{\left\{X_0-E[X_0|\mathcal{F}_{-l}^{l}]\geq \sqrt{a_l}\right\}}\right]\\
&+E\left[\left(g\left(X_0\right)-g\left(E[X_0|\mathcal{F}_{-l}^{l}]\right)\right)\mathds{1}_{\left\{X_0-E[X_0|\mathcal{F}_{-l}^{l}]< \sqrt{a_l}\right\}}\right]\displaybreak[0]\\
\leq&E\left[\sup_{\left\|x-X_0\right\|\leq \sqrt{a_l},\ \left\|x'-X_0\right\|\leq \sqrt{a_l}}\left|g\left(x\right)-g\left(x'\right)\right|\right]+KP\left[X_0-E[X_0\big|\mathcal{F}_{-l}^{l}]\geq \sqrt{a_l}\right]\\
\leq&L\sqrt{a_l}+K\sqrt{a_l}.
\end{align*}

\end{proof}

\begin{lem}\label{lem5} Let $\left(X_n\right)_{n\in\N}$ be an $1$-approximating functional of an absolutely regular process $\left(Z_n\right)_{n\in\mathds{Z}}$ with mixing coefficients $\beta(n)=O\left(n^{-\beta}\right)$ for a $\beta>3$ and approximation constants $a_n=O\left(n^{-(\beta+3)}\right)$. Let $C_1,C_2,L>0$ be constants. Then there exists a constant $C$, such that for all measurable, nonnegative functions $g:\R\rightarrow\R$ that are bounded by $C_1$ with $E\left|g\left(X_1\right)-Eg\left(X_1\right)\right|\geq C_2n^{-\frac{\beta}{\beta+1}}$ and satisfy the variation condition with constant $L$, and all $n\in\N$ we have
\begin{equation*}
 E\left(\sum_{i=1}^{n}g\left(X_i\right)-E\left[g\left(X_1\right)\right]\right)^4\leq C n^2\left(\log n\right)^2\left(E\left|Y_1\right|\right)^{1+\gamma}
\end{equation*}
with $\gamma=\frac{\beta-3}{\beta+1}$.
\end{lem}

\begin{proof}
We define the random variables $Y_i=g\left(X_i\right)-Eg\left(X_1\right)$. Then by Lemma \ref{lem6}, $\left(Y_n\right)_{n\in\N}$ is an $1$-approximating functional with approximation constants $\tilde{a}_n=(L+C_1)\sqrt{a_n}=O\left(n^{-\frac{\beta+3}{2}}\right)$. Using Lemma \ref{lem3} with $\delta=\frac{6}{\beta-3}$, we obtain
\begin{multline*}
\left|EY_0Y_iY_{i+j}Y_{i+j+k}\right|\\
\leq C\left(\beta^{\frac{3}{\beta}}\left(\lfloor\frac{\max\left\{i,j,k\right\}}{3}\rfloor\right)\left\|Y_0\right\|_{\frac{2\beta}{\beta-3}}^2+\tilde{a}^{\frac{6}{\beta+3}}_{\lfloor\frac{\max\left\{i,j,k\right\}}{3}\rfloor}\left\|Y_0\right\|_{\frac{2\beta}{\beta-3}}^{\frac{2\beta}{\beta+3}}\right)+\left|E\left[Y_0Y_i\right]E\left[Y_0Y_k\right]\right|.
\end{multline*}
Making use of Lemma \ref{lem2} and $\delta=\frac{2}{\beta-1}$, it follows that
\begin{multline*}
\left|EY_0Y_iY_{i+j}Y_{i+j+k}\right|\leq C\left(\beta^{\frac{3}{\beta}}\left(\lfloor\frac{\max\left\{i,j,k\right\}}{3}\rfloor\right)\left\|Y_0\right\|_{\frac{2\beta}{\beta-3}}^2+\tilde{a}^{\frac{6}{\beta+3}}_{\lfloor\frac{\max\left\{i,j,k\right\}}{3}\rfloor}\left\|Y_0\right\|_{\frac{2\beta}{\beta-3}}^{\frac{2\beta}{\beta+3}}\right)\\
+C\left(\beta^{\frac{1}{\beta}}\left(\lfloor \frac{k}{3}\rfloor\right)\left\|Y_0\right\|^2_{\frac{2\beta}{\beta-1}}+\tilde{a}^{\frac{2}{\beta+1}}_{\lfloor \frac{k}{3}\rfloor}\left\|Y_0\right\|^{\frac{2\beta}{\beta+1}}_{\frac{2\beta}{\beta-1}}\right)\cdot\left(\beta^{\frac{1}{\beta}}\left(\lfloor \frac{i}{3}\rfloor\right)\left\|Y_0\right\|^2_{\frac{2\beta}{\beta-1}}+\tilde{a}^{\frac{2}{\beta+1}}_{\lfloor \frac{i}{3}\rfloor}\left\|Y_0\right\|^{\frac{2\beta}{\beta+1}}_{\frac{2\beta}{\beta-1}}\right).
\end{multline*}
First note that
\begin{align*}
\beta^{\frac{1}{\beta}}(n)=O\left(n^{-1}\right),\ \ \ &\tilde{a}^{\frac{2}{\beta +1}}=O\left(n^{-1}\right),\\
\beta^{\frac{3}{\beta}}(n)=O\left(n^{-3}\right),\ \ \ &\tilde{a}^{\frac{6}{\beta +3}}=O\left(n^{-3}\right),
\end{align*}
and that
\begin{align*}
\left\|Y_0\right\|^2_{\frac{2\beta}{\beta-1}}&\leq C\left\|Y_0\right\|^{\frac{2\beta}{\beta+1}}_{\frac{2\beta}{\beta-1}}\leq C\left\|Y_0\right\|_1^{\frac{\beta-1}{\beta+1}},\\
\left\|Y_0\right\|_{\frac{2\beta}{\beta-3}}^2&\leq C\left\|Y_0\right\|_{\frac{2\beta}{\beta-3}}^{\frac{2\beta}{\beta+3}}\leq C\left\|Y_0\right\|_1^{\frac{\beta-3}{\beta+3}}\leq Cn\left\|Y_0\right\|^{\frac{2\beta-2}{\beta+1}}_1,
\end{align*}
as $E\left|Y_1\right|\geq C_2n^{-\frac{\beta}{\beta+1}}$. Now by stationarity
\begin{align*}
 &E\left(\sum_{i=1}^{n}Y_i\right)^4\leq Cn\sum_{i,j,k=1}^{n}\left|E\left[Y_0Y_iY_{i+j}Y_{i+j+k}\right]\right|\\
\leq&Cn^2\sum_{i=1}^{n}\beta^{\frac{1}{\beta}}\left(\lfloor \frac{i}{3}\rfloor\right)\sum_{k=1}^n\beta^{\frac{1}{\beta}}\left(\lfloor \frac{k}{3}\rfloor\right)\left\|Y_1\right\|_1^{\frac{2\beta-2}{\beta}}+Cn^2\sum_{i=1}^n\tilde{a}^{\frac{2}{\beta+1}}_{\lfloor \frac{i}{3}\rfloor}\sum_{k=0}^n\tilde{a}^{\frac{2}{\beta+1}}_{\lfloor \frac{k}{3}\rfloor}\left\|Y_1\right\|_1^{\frac{2\beta-2}{\beta+1}}\\
&+Cn\sum_{m=1}^nm^2\beta^{\frac{3}{\beta}}\left(\lfloor\frac{m}{3}\rfloor\right)\left\|Y_0\right\|_1^{\frac{\beta-3}{\beta}}+Cn\sum_{m=1}^nm^2\tilde{a}^{\frac{6}{\beta+3}}_{\lfloor\frac{m}{3}\rfloor}\left\|Y_0\right\|_1^{\frac{\beta-3}{\beta+3}}\displaybreak[0]\\
\leq&Cn^2\sum_{i=1}^{n}i^{-1}\sum_{k=1}^{n}k^{-1}\left\|Y_1\right\|_1^{\frac{2\beta-2}{\beta+1}}
+Cn^2\sum_{m=1}^{n}m^2m^{-3}\left\|Y_1\right\|_1^{\frac{2\beta-2}{\beta+1}}\\
\leq&C n^2\left(\log n\right)^2\left(E\left|Y_1\right|\right)^{\frac{2\beta-2}{\beta+1}}=C n^2\left(\log n\right)^2\left(E\left|Y_1\right|\right)^{1+\gamma}
\end{align*}

\end{proof}

We use the representation $R_n=Z_n\left(F\left(t_p\right)-F_n\left(t_p\right)\right)$, so we have to know the a.s. asymptotic behaviour of $F_n\left(t_p\right)-F\left(t_p\right)$. The law of the iterated logarithm for functionals of mixing data has been proved by Reznik \cite{rezn}. We only prove that $\sqrt{\frac{n}{\log \log n}}\left(F_n\left(t_p\right)-F\left(t_p\right)\right)$ is bounded a.s., but under somewhat milder conditions, which fit better to our theorems:

\begin{prop}\label{theo3b} Let $\left(X_n\right)_{n\in\N}$ be a bounded, $1$-approximating functional with approximation constants $a_n=O\left(n^{-\beta}\right)$ for some $\beta>3$ of an absolutely regular process 
$\left(Z_n\right)_{n\in\N}$ with mixing coefficients $\beta\left(n\right)=O\left(n^{-\beta}\right)$. Then
\begin{equation}
 \sum_{i=1}^{n}\left(X_i-EX_i\right)=O\left(\sqrt{n\log\log n}\right)\quad\text{a.s.}
\end{equation}
\end{prop}

\begin{proof} W.l.o.g. we assume that $EX_i=0$. We use a blocking technique and define
\begin{equation*}
 B_{in}=\sum_{j=1}^{k}X_{(i-1)k+j}
\end{equation*}
for $i=1,\ldots,\lfloor\frac{n}{k}\rfloor$ with $k=k_n=\lfloor\frac{2^{\frac{l}{2}}}{\log l}\rfloor$ for $2^{l}\leq n<2^{l+1}$ and write
\begin{equation*}
\sum_{i=1}^{n}X_i=\sum_{\substack{s\leq\lfloor\frac{n}{k}\rfloor\\s\text{ odd}}}B_{sn}+
\sum_{\substack{s\leq\lfloor\frac{n}{k}\rfloor\\s\text{ even}}}B_{sn}+\sum_{i=k\lfloor\frac{n}{k}\rfloor+1}^nX_i.
\end{equation*}
By Lemma 2.24 of Borovkova et al. \cite{boro}, we have that for all $N,m\in\N$
\begin{equation*}
E\left(\sum_{i=N+1}^{N+m}X_i\right)^4\leq Cm^2
\end{equation*}
and by Corollary 1 of M\'oricz \cite{mori} that
\begin{equation*}
E\left(\max_{1\leq m\leq k}\left|\sum_{i=k\lfloor\frac{n}{k}\rfloor+1}^{k\lfloor\frac{n}{k}\rfloor+m}X_i\right|\right)^4\leq Ck^2.
\end{equation*}
It follows that
\begin{equation*}
E\left(\max_{2^l\leq n< 2^{l+1}}\left|\sum_{i=k\lfloor\frac{n}{k}\rfloor+1}^nX_i\right|\right)^4\leq\frac{n}{k}E\left(\max_{1\leq m\leq k}\left|\sum_{i=k\lfloor\frac{n}{k}\rfloor+1}^{k\lfloor\frac{n}{k}\rfloor+m}X_i\right|\right)^4\leq Cnk.
\end{equation*}
So we get for every $\epsilon>0$
\begin{multline*}
\sum_{l=0}^{\infty}P\left[\max_{2^l\leq n< 2^{l+1}}\left|\sum_{i=k\lfloor\frac{n}{k}\rfloor+1}^nX_i\right|\geq 2^{\frac{l}{2}}\epsilon\right]\\
\leq \sum_{l=0}^{\infty}\frac{1}{\epsilon^42^{2l}}E\left(\max_{2^l\leq n< 2^{l+1}}\left|\sum_{i=k\lfloor\frac{n}{k}\rfloor+1}^nX_i\right|\right)^4\leq\frac{C}{\epsilon^4}\sum_{l=0}^{\infty}\frac{2^{\frac{3}{2}l}\log l}{2^{2l}}<\infty
\end{multline*}
and by a applying the Borel-Cantelli lemma we conclude that $\sum_{i=k\lfloor\frac{n}{k}\rfloor+1}^nX_i=o\left(\sqrt{n}\right)$ a.s. By Theorem 3 of Borovkova et al. \cite{boro}, there exists a sequence of independent random variables $\left(B'_{sn}\right)_{s\in\N}$, such that for all even $s$
\begin{equation*}
P\left[\left|B_{sn}-B'_{sn}\right|\leq 2A_{\lfloor\frac{k}{3}\rfloor}\right]\geq 1-2A_{\lfloor\frac{k}{3}\rfloor}-\beta_{\lfloor\frac{k}{3}\rfloor}
\end{equation*}
with $A_L=\sqrt{2\sum_{l=L}^{\infty}a_l}=O\left(L^{-(1+\frac{\beta-3}{2})}\right)$. It follows that
\begin{multline*}
P\left[\sup_{m\leq\lfloor\frac{2^{l+1}}{k}\rfloor}\sum_{\substack{1\leq s\leq m\\s\text{ even}}}\left|B_{sn}-B'_{sn}\right|\geq 2\frac{n}{k}A_{\lfloor\frac{k}{3}\rfloor} \right]\leq \frac{2^{l+1}}{2k}\left(2A_{\lfloor\frac{k}{3}\rfloor}+\beta_{\lfloor\frac{k}{3}\rfloor}\right)\\
\leq C\frac{2^{l+1}}{k^{2+\frac{\beta-3}{2}}}\leq C2^{-\frac{\beta-3}{4}l}(\log l)^{\frac{\beta+1}{2}}.
\end{multline*}
Note that $2\frac{n}{k}A_{\lfloor\frac{k}{3}\rfloor}\rightarrow0$ as $n\rightarrow\infty$ so that
\begin{equation*}
\sum_{l=1}^{\infty}P\left[\sup_{2^l\leq n<2^{l+1}}\left|\sum_{\substack{s\leq\lfloor\frac{n}{k}\rfloor\\s\text{ even}}}B_{sn}-B_{sn}'\right|\geq \epsilon\right]\leq C\sum_{l=1}^{\infty}2^{-l\frac{\beta-3}{4}}(\log l)^{\frac{\beta+1}{2}}<\infty
\end{equation*}
and it follows hat $\sum_{\substack{s\leq\lfloor\frac{n}{k}\rfloor\\s\text{ even}}}\left(B_{sn}-B_{sn}'\right)=o\left(1\right)$ a.s. The same arguments justify that there exists sequences $\left(B_{(sn)}''\right)_{s\in\N}$, such that  $\sum_{\substack{s\leq\lfloor\frac{n}{k}\rfloor\\s\text{ odd}}}\left(B_{sn}-B_{sn}''\right)=o\left(1\right)$, so it suffices to show that
\begin{equation*}
\frac{1}{\sqrt{n\log\log n}}\left|\sum_{\substack{s\leq\lfloor\frac{n}{k}\rfloor\\s\text{ even}}}B_{sn}'\right|\leq C
\end{equation*}
a.s. (the sequences $\left(B_s''\right)$ can be treated in the same way). By Lemma 2.23 of Borovkova et al. \cite{boro}, we have that
\begin{equation*}
\var\left[B_{sn}'\right]\leq Ck
\end{equation*}
and
\begin{equation*}
\sum_{\substack{s\leq\lfloor\frac{n}{k}\rfloor\\s\text{ even}}}\var\left[B_{sn}'\right]\leq Cn
\end{equation*}
and additionally $\left|B_{sn}'\right|\leq Ck$. So by Bernstein's inequality (see Bennett \cite{benn}), we obtain for all $N\leq\lfloor\frac{n}{k}\rfloor$, $2^l\leq n<2^{l+1}$ and $C_1>0$
\begin{multline*}
P\left[\left|\sum_{\substack{N\leq s\leq\lfloor\frac{2^{(l+1)}}{k}\rfloor\\s\text{ even}}}B_{sn}'\right|\geq C_1\sqrt{2^l\log l}\right]\leq2e^{-\frac{C_1^2 2^l\log l}{-2\sum\var\left[B_{sn}'\right]+2C_1\sqrt{2^l\log l}\left\|B_{1n}'\right\|_{\infty}}}\\
\leq2e^{-\frac{C_1^2 2^l\log l}{C2^{(l+1)}+CC_1\sqrt{2^l\log l}\lfloor 2^\frac{l}{2}\log^{-1} l\rfloor}}\leq2l^{-\frac{C_1}{C}}.
\end{multline*}
Due to Skorohod's inequality (see Shorack, Wellner \cite{shor}, p. 844), we conclude that
\begin{equation}\label{line27}
P\left[\sup_{2^l\leq n<2^{l+1}}\left|\sum_{\substack{ s\leq\lfloor\frac{n}{k}\rfloor\\s\text{ even}}}B_s'\right|\geq 2C_1\sqrt{n\log\log n}\right]\leq \frac{2l^{-\frac{C_1}{C}}}{1-2l^{-\frac{C_1}{C}}}.
\end{equation}
Choosing the constant $C_1$ large enough, the probabilities in Line (\ref{line27}) are summable and
\begin{equation*}
\frac{1}{\sqrt{n\log\log n}}\left|\sum_{\substack{s\leq\lfloor\frac{n}{k}\rfloor\\s\text{ even}}}B_{sn}'\right|\leq 2C_1
\end{equation*} 
for almost all $n\in\N$ a.s. follows by the Borel-Cantelli lemma.
\end{proof}

\subsection{U-Quantiles}

$U$-statistics can be decomposed into a linear and a degenerate part, which is a $U$-statistic with kernel $h_2(x,y,t):=h(x,y,t) - h_1(x,t) -h_1(y,t) -U\left(t\right)$. If $h$ is bounded and satisfies the variation condition in $t$, the same holds for $h_2$, see Lemma 4.5 of Dehling, Wendler \cite{deh2}. Furthermore, $h_2$ is degenerate, i.e. for all $y,t\in\R:$ $Eh_2\left(X_1,y,t\right)=0$. For the degenerate part, we need generalized covariance inequalities.

\begin{lem}\label{lem7} Let $\left(X_n\right)_{n\in\N}$ be a stationary, strongly mixing sequence with $\left\|X_n\right\|_1<\infty$, $h:\R\times\R\times\R\rightarrow\R$ a bounded kernel function that satisfies the variation condition in $t$. Then there is a constant, such that
\begin{equation*}
 \left|E\left|h_2\left(X_{i_1},X_{i_2},t\right)h_2\left(X_{i_3},X_{i_4},t\right)\right]\right|\leq C\alpha^{\frac{1}{2}}\left(m\right),
\end{equation*}
where $m=\max\left\{i_{(2)}-i_{(1)},i_{(4)}-i_{(3)}\right\}$, $\left\{i_1,i_2,i_3,i_4\right\}=\left\{i_{(1)},i_{(2)}, i_{(3)},i_{(4)}\right\}$ and $i_{(1)}\leq i_{(2)}\leq i_{(3)}\leq i_{(4)}$.
\end{lem}

\begin{proof} The result is easily obtained by taking the limit $\delta\rightarrow\infty$ in Lemma 4.2 of Dehling, Wendler \cite{deh2}.
 \end{proof}

\begin{lem}\label{lem8} Let $\left(X_n\right)_{n\in\N}$ be an $1$-approximating functional with approximation constants $\left(a_n\right)_{n\in\N}$ of an absolutely regular process with mixing coefficients $\left(\beta(k)\right)_{k\in\N}$. Let $h\left(\cdot,\cdot,t\right):\R\times\R\rightarrow\R$ be a bounded kernel function that satisfies the variation conditon in $t$. Then
\begin{equation*}
 \left|E\left[h_2\left(X_{i_1},X_{i_2},t\right)h_2\left(X_{i_3},X_{i_4},t\right)\right]\right|\leq C\left(\beta(\lfloor\frac{m}{3}\rfloor)+A_{\lfloor\frac{m}{3}\rfloor}\right)
\end{equation*}
with $A_L=\sqrt{2\sum_{l=L}^{\infty}a_l}$.
\end{lem}

\begin{proof} The result is easily obtained by taking the limit $\delta\rightarrow\infty$ in Lemma 4.3 of Dehling, Wendler \cite{deh2}.
 \end{proof}

\begin{lem}\label{lem9} If a kernel function $h:\R\times\R\times\R\rightarrow\R$ satisfies the variation condition in $t$ with constant $L$, then the variation condition holds for $h_1(\cdot,t)$ with the same constant $L$. 
\end{lem}

\begin{proof}
Let be $Y$ independent of $X$ with the same distribution as $X$. Then
\begin{multline*}
 E\left[\sup_{\left\|x-X\right\|\leq \epsilon,\ \left\|x'-X\right\|\leq \epsilon}\left|h_1\left(x,t\right)-h_1\left(x',t\right)\right|\right]\\
= E\left[\sup_{\left\|x-X\right\|\leq \epsilon,\ \left\|x'-X\right\|\leq \epsilon}\left|Eh\left(x,Y,t\right)-Eh\left(x',Y,t\right)\right|\right]\displaybreak[0]\\
\leq E\left[\sup_{\left\|x-X\right\|\leq \epsilon,\ \left\|x'-X\right\|\leq \epsilon}\left|h\left(x,Y,t\right)-h\left(x',Y,t\right)\right|\right]\\
\leq E\left[\sup_{\left\|(x,y)-(X,Y)\right\|\leq \epsilon,\ \left\|(x',y')-(X,Y)\right\|\leq \epsilon}\left|h\left(x,y,t\right)-h\left(x',y',t\right)\right|\right]\leq L\epsilon.
\end{multline*}

\end{proof}

The law of the iterated logarithm for $U$-statistics has been investigated by Dehling, Wendler \cite{deh2}, but here we state it under slightly different conditions:

\begin{prop}\label{theo4} Let $\left(X_n\right)_{n\in\N}$ be a stationary, strongly mixing sequence with $\left\|X_n\right\|_1<\infty$, $h:\R\times\R\times\R\rightarrow\R$ a bounded kernel function which satisfies the variation condition in $t$. If the mixing coefficients satisfy $\alpha\left(n\right)=O\left(n^{-\beta}\right)$ for some $\beta>2$, then a.s.
\begin{equation}
\limsup_{n\rightarrow\infty}\pm\sqrt{\frac{n}{\log\log n}}U_n\left(t\right)=\sqrt{2\sigma_1^2}
\end{equation}
with $\sigma^2_1=\var\left[h_1\left(X_1,t\right)\right]+2\sum_{k=2}^\infty\cov\left[h_1\left(X_1,t\right),h_1\left(X_k,t\right)\right]$.
\end{prop}

\begin{proof}
The proof is the same as the proof of Theorem 2 of Dehling, Wendler \cite{deh2}, where Lemma \ref{lem7} playes the role of Lemma 4.2 of Dehling, Wendler \cite{deh2}, and hence omitted.
\end{proof}

For functionals of absolutely regular sequences, we give not the full law of the iterated logarithm, only a weaker version under much milder conditions than in Dehling, Wendler \cite{deh2}.

\begin{prop}\label{theo4b} Let $\left(X_n\right)$ be an $1$-approximating functional with approximation constants $a_n=O\left(n^{-(\beta+3)}\right)$ for some $\beta>3$ of an absolutely regular process $\left(Z_n\right)_{n\in\Z}$ with mixing coefficients $\beta(n)=O\left(n^{-\beta}\right)$. Let $h:\R\times\R\times\R\rightarrow\R$ be a bounded kernel function which satisfies the varitation condition in $t$. Then
\begin{equation}
\left(U_n(t)-EU_n(t)\right)=O\left(\sqrt{\frac{\log\log n}{n}}\right)\ \ \ \text{a.s.}
\end{equation}
\end{prop}

\begin{proof} We use the Hoeffding decomposition
\begin{equation*}
U_n\left(t\right)-EU_n\left(t\right)=\frac{2}{n}\sum_{i=1}^{n}h_{1}\left(X_{i},t\right)+\frac{2}{n\left(n-1\right)}\sum_{1\leq i<j\leq n}h_{2}\left(X_{i},X_{j},t\right).
\end{equation*}
Note that $h_1$ satsifies the $1$-approximation condition in $t$ by Lemma \ref{lem9} and by Lemma \ref{lem6} $\left(h_1\left(X_n,t\right)\right)_{n\in\N}$ is an $1$-approximating functional of $\left(Z_n\right)_{n\in\Z}$ with approximation constants $C\sqrt{a_n}=O\left(n^{-\frac{\beta+3}{2}}\right)$, so by Proposition \ref{theo3b}
\begin{equation*}
\frac{2}{n}\sum_{i=1}^{n}h_{1}\left(X_{i},t\right)=O\left(\sqrt{\frac{\log\log n}{n}}\right)\ \ \ \text{a.s.}
\end{equation*}
With Lemma \ref{lem8} replacing Lemma 4.3 of Dehling, Wendler \cite{deh2} we can prove in similarly to Theorem 1 of Dehling, Wendler \cite{deh2} that
\begin{equation*}
\frac{2}{n\left(n-1\right)}\sum_{1\leq i<j\leq n}h_{2}\left(X_{i},X_{j},t\right)=o\left(\frac{(\log n)^{\frac{3}{2}}\log\log n}{n}\right)
\end{equation*}
a.s., which completes the proof.
\end{proof}

Borovkova et al. \cite{boro} and Dehling, Wendler \cite{dehl} have established the central limit theorem for $U$-statistics under $p$-continuity, which is a similar assumption to the variation condition. The central limit theorem still holds under the variation condition:

\begin{prop}\label{theo5} Let $h:\R\times\R\times\R\rightarrow\R$ be a bounded kernel function that satisfies the variation condition in $t$ and let one of the following two mixing conditions hold:
\begin{enumerate}
\item Let $\left(X_n\right)_{n\in\N}$ be a strongly mixing sequence with $E\left|X_1\right|<\infty$, and $\alpha(n)=O\left(n^{-\beta}\right)$ for a $\beta>2$.
\item Let $\left(X_n\right)$ be a $1$-approximating functional with approximation constants $a_n=O\left(n^{-(\beta+3)}\right)$ for some $\beta>3$ of an absolutely regular process $\left(Z_n\right)_{n\in\Z}$ with mixing coefficients $\beta(n)=O\left(n^{-\beta}\right)$.
\end{enumerate}
Then
\begin{equation}
 \sqrt{n}\left(U_n\left(t\right)-U(t)\right)\xrightarrow{\mathcal{D}}N\left(0,\sigma^2_1\right)
\end{equation}
with
\begin{equation*}
\sigma^2_1=\var\left[h_1\left(X_1,t\right)\right]+2\sum_{k=2}^\infty\cov\left[h_1\left(X_1,t\right),h_1\left(X_k,t\right)\right].
\end{equation*}
\end{prop}

\begin{proof} Under Condition 1. the proof is the same as for Theorem 1.8 of Dehling, Wendler \cite{dehl} with our Lemma \ref{lem7} replacing their Lemma 3.3. Under Condition 2., we replace Lemma 4.3 of Borovkova et al. \cite{boro} by our Lemma \ref{lem8} in the proof of their Theorem 7.
\end{proof}

\section{Proof of Main results}

\subsection{Sample Quantiles}

In the proofs, $C$ denotes an arbitrary constant, which may have different values from line to line and may depend on several other values, but not on $n\in\N$.

\begin{proof}[Proof of Theorem \ref{theo1}] Let $c_n=n^{-\frac{5}{8}-\frac{1}{8}\gamma}(\log n)^{\frac{3}{4}}(\log\log n)^{\frac{1}{2}}$. We first prove that
\begin{multline*}
\sum_{l=0}^{\infty}P\left[\max_{2^l\leq n<2^{l+1}}\frac{1}{c_n}\sup_{|t-t_p|\leq C\sqrt{\frac{\log l}{2^l}}}\left(F_n\left(t\right)-F_n\left(t_p\right)-F\left(t\right)+F\left(t_p\right)\right)>\epsilon\right]\\
\leq C\sum_{l=0}^{\infty}\frac{1}{c_{2^l}^4}E\left(\max_{2^l\leq n<2^{l+1}}\sup_{|t-t_p|\leq C\sqrt{\frac{\log l}{2^l}}}\left(F_n\left(t\right)-F_n\left(t_p\right)-F\left(t\right)+F\left(t_p\right)\right)\right)^4<\infty.
\end{multline*}
Line (\ref{line10}) will follow by the Borel-Cantelli lemma. We set $d_{2^l}=\left(\frac{\log l}{2^l}\right)^{\frac{3}{4}}$ and $d_n=d_{2^l}$ for $2^l\leq n<2^{l+1}$. Let $k\in\Z$. As $F_n$, $F$ are nondecreasing in $t$, we have for any $t\in\left[t_p+kd_n,t_p+(k+1)d_n\right]$ that
\begin{align*}
&\left|F_n\left(t\right)-F_n\left(t_p\right)-F\left(t\right)+F\left(t_p\right)\right|\\
\leq&\max\left\{\left|F_n\left(t_p+kd_n\right)-F_n\left(t_p\right)-F\left(t\right)+F\left(t_p\right)\right|\right.,\\
&\quad\left.\left|F_n\left(t_p+(k+1)d_n\right)-F_n\left(t_p\right)-F\left(t_p\right)+F\left(t_p\right)\right|\right\}\displaybreak[0]\\
\leq&\max\left\{\left|F_n\left(t_p+kd_n\right)-F_n\left(t_p\right)-F\left(t_p+kd_n\right)+F\left(t_p\right)\right|\right.,\\
&\quad\left.\left|F_n\left(t_p+(k+1)d_n\right)-F_n\left(t_p\right)-F\left(t_p+(k+1)d_n\right)+F\left(t_p\right)\right|\right\}\\
&+\left|F\left(t_p+(k+1)d_n\right)-F\left(t_p+kd_n\right)\right|.
\end{align*}
It follows that
\begin{align*}
&\sup_{|t-t_p|\leq C\sqrt{\frac{\log l}{2^l}}}\left(F_n\left(t\right)-F_n\left(t_p\right)-F\left(t\right)+F\left(t_p\right)\right)\\
\leq&\max_{\left|k\right|\leq C\left(2^l\log l\right)^{\frac{1}{4}}}\left(F_n\left(t_p+d_nk\right)-F_n\left(t_p\right)-F\left(t_p+d_nk\right)+F\left(t_p\right)\right)\\
&+\max_{\left|k\right|\leq C\left(2^l\log l\right)^{\frac{1}{4}}}\left|F\left(t_p+(k+1)d_n\right)-F\left(t_p+kd_n\right)\right|.
\end{align*}
From condition (\ref{line9}), we conclude that
\begin{equation*}
\max_{\left|k\right|\leq C\left(2^l\log l\right)^{\frac{1}{4}}}\left|F\left(t_p+(k+1)d_n\right)-F\left(t_p+kd_n\right)\right|\leq
f(t_p)d_n+o\left(\left(\sqrt{\frac{\log l}{2^l}}\right)^{\frac{3}{2}}\right)=o\left(c_n\right).
\end{equation*}
Furthermore, we have that for all $k_1,k_2\leq C\left(2^l\log l\right)^{\frac{1}{4}}$
\begin{equation*}
 \left|F\left(t_p+d_nk_1\right)-F\left(t_p+d_nk_2\right)\right|=f\left(t_p\right)\left|k_1-k_2\right|d_n+o\left(\sqrt{\frac{\log l}{2^l}}^{\frac{3}{2}}\right)\leq C\left|k_1-k_2\right|d_n.
\end{equation*}
So by Lemma \ref{lem4} (under mixing Condition 1.) or Lemma \ref{lem5} (under mixing Condition 2.)
\begin{multline*}
E\left(F_n\left(t_p+d_nk_1\right)-F_n\left(t_p+d_nk_2\right)-F\left(t_p+d_nk_1\right)+F\left(t_p+d_nk_2\right)\right)^4\\
\leq C\frac{1}{n^2}\left(\log n\right)^2\left|k_1-k_2\right|^{1+\gamma}d_n^{1+\gamma}.
\end{multline*}
Note that we can represent the differences of the empirical distribution function as a double sum
\begin{multline*}
F_n\left(t_p+d_nk\right)-F_n\left(t_p\right)-F\left(t_p+d_nk\right)+F\left(t_p\right)\\
=\sum_{i=1}^{n}\sum_{j=1}^k\left(g(X_i,t_p+jd_n)-g(X_i,t_p+(j-1)d_n)-F(t_p+jd_n)+F(t_p+(j-1)d_n)\right),
\end{multline*}
so by Corollary 1 of M\'oricz \cite{mori}, it then follows that
\begin{multline*}
\frac{1}{c_{2^l}^4}E\left(\max_{2^l\leq n<2^{l+1}}\max_{\left|k\right|\leq C\left(2^l\log l\right)^{\frac{1}{4}}}\left(F_n\left(t_p+d_nk\right)-F_n\left(t_p\right)-F\left(t_p+d_nk\right)+F\left(t_p\right)\right)\right)^4\\
\shoveleft\leq C\frac{1}{c_{2^l}^4}E\left(F_n\left(t_p+C\sqrt{\frac{\log \log n}{n}}\right)-F_n\left(t_p-C\sqrt{\frac{\log \log n}{n}}\right)\right.\\
\left.-F\left(t_p+C\sqrt{\frac{\log \log n}{n}}\right)+F\left(t_p-C\sqrt{\frac{\log\log n}{n}}\right)\right)^4\\
\shoveleft\leq C\frac{2^{\frac{5+\gamma}{2}l}}{l^{3}\left(\log l\right)^{2}}\frac{l^2}{2^{2l}}\frac{\left(\log l\right)^{\frac{1+\gamma}{2}}}{2^{\frac{1+\gamma}{2}l}}=C\frac{1}{l\left(\log l\right)^{\frac{3-\gamma}{2}}}.\\
\end{multline*}
As $\gamma<1$, this quantities are summable and Line (\ref{line10}) is proved.

To prove Line (\ref{line11}), let w.l.o.g. $f\left(t_p\right)=1$, otherwise replace $g\left(x,t\right)$ by $g\left(x,\frac{t}{f(t_p)}\right)$. We represent $R_n$ as $Z_n\left(F\left(t_p\right)-F_n\left(t_p\right)\right)$ with
\begin{equation*}
 Z_n\left(x\right):=\left(F_n\left(\cdot+t_p\right)-F_n\left(t_p\right)\right)^{-1}\left(x\right)-x=F_n^{-1}\left(x+F_n\left(t_p\right)\right)-x-t_p.
\end{equation*}
By Theorem 3 of Rio \cite{rio} respectively Proposition \ref{theo3b} a.s.
\begin{equation*}
 \limsup_{n\rightarrow\infty}\pm\sqrt{\frac{n}{\log\log n}}\left(F_n\left(t_p\right)-F\left(t_p\right)\right)\leq C.
\end{equation*}
By Line (\ref{line10}) and Condition (\ref{line9})
\begin{align*}
&\sup_{\left|x\right|\leq C\sqrt{\frac{\log\log n}{n}}}\left|F_n\left(x+t_p\right)-F_n\left(t_p\right)-x\right|\\
=&\sup_{\left|x\right|\leq C\sqrt{\frac{\log\log n}{n}}}\left|F_n\left(x+t_p\right)-F\left(x+t_p\right)-F_n\left(t_p\right)+F\left(t_p\right)\right|\\
&+\sup_{\left|x\right|\leq C\sqrt{\frac{\log\log n}{n}}}\left|F\left(x+t_p\right)-F\left(t_p\right)-x\right|=o\left(c_n\right)\ \ \ \text{a.s.}
\end{align*}
Then by Theorem 1 of Vervaat \cite{verv}
\begin{equation*}
\sup_{\left|x\right|\leq C\sqrt{\frac{\log\log n}{n}}}\left|Z_n\left(x\right)\right|=o\left(c_n\right)\ \ \ \text{a.s.},
\end{equation*}
(Vervaats theorem is for random functions from $[0,\infty)$ to $[0,\infty)$, but it becomes clear from the proof of his Lemma 1 that it also holds for the intervalls $[-C\sqrt{\frac{\log\log n}{n}},C\sqrt{\frac{\log\log n}{n}}]$). Hence $R_n=Z_n\left(F\left(t_p\right)-F_n\left(t_p\right)\right)=o\left(c_n\right)$ a.s.
\end{proof}

\subsection{$U$-Quantiles}

\begin{proof}[Proof of Theorem \ref{theo2}] To prove Line (\ref{line16}), we use the Hoeffding decomposition
\begin{equation*}
U_n\left(t\right)=U\left(t\right)+\frac{2}{n}\sum_{i=1}^{n}h_{1}\left(X_{i},t\right)+\frac{2}{n\left(n-1\right)}\sum_{1\leq i<j\leq n}h_{2}\left(X_{i},X_{j},t\right).
\end{equation*}
As above, we set $c_n=n^{-\frac{5}{8}-\frac{1}{8}\gamma}(\log n)^{\frac{3}{4}}(\log\log n)^{\frac{1}{2}}$ and $d_n=\left(\frac{\log\log n}{n}\right)^{\frac{3}{4}}$ and get
\begin{align*}
&\sup_{|t-t_p|\leq C\sqrt{\frac{\log l}{2^l}}}\left|U_n\left(t\right)-U_n\left(t_p\right)-U\left(t\right)+U\left(t_p\right)\right|\\
\leq&\max_{\left|k\right|\leq C\left(2^l\log l\right)^{\frac{1}{4}}}\left|U_n\left(t_p+d_nk\right)-U_n\left(t_p\right)-U\left(t_p+d_nk\right)+U\left(t_p\right)\right|\\
&+\max_{\left|k\right|\leq C\left(2^l\log l\right)^{\frac{1}{4}}}\left|U\left(t_p+d_n(k+1)\right)-U\left(t_p+d_nk\right)\right|
\end{align*}
and
\begin{equation*}
\max_{\left|k\right|\leq C\left(2^l\log l\right)^{\frac{1}{4}}}\left|U\left(t_p+d_n(k+1)\right)-U\left(t_p+d_nk\right)\right|=o\left(c_n\right).
\end{equation*}

By Lemma \ref{lem9} we have that $h_1$ satisfies the variation condition uniformly in some neighbourhood of $t_p$. Applying Theorem \ref{theo1} to the function $g=h_1$, we obtain
\begin{multline*}
\max_{\left|k\right|\leq C\left(2^l\log l\right)^{\frac{1}{4}}}\left|\frac{2}{n}\sum_{i=1}^{n}h_{1}\left(X_{i},t_p+kd_n\right)-\frac{2}{n}\sum_{i=1}^{n}h_{1}\left(X_{i},t_p\right)-U\left(t_p+d_nk\right)+U\left(t_p\right)\right|\\
=o\left(c_n\right)
\end{multline*}
a.s. It remains to show that
\begin{equation}\label{line31}
\max_{\left|k\right|\leq C\left(2^l\log l\right)^{\frac{1}{4}}}\left|Q_n\left(t_p+d_nk\right)-Q_n\left(t_p\right)\right|=o\left(n^2c_n\right)
\end{equation}
a.s. with $Q_n\left(t\right):=\sum_{1\leq i<j\leq n}h_{2}\left(X_{i},X_{j},t\right)$. We first consider Condition 1. (strong mixing) and concentrate on the case $\beta<4$. In the case $\beta\geq 4$, a similar calculation can be done. Recall that for any random variables $Y_1,\ldots,Y_m$: $E\left(\max_{i=1,\ldots,m}|Y_i|\right)^2\leq\sum_{i=1}^mEY_i^2$ and therefore
\begin{multline*}
 E\left(\max_{2^{l-1}\leq n<2^l}\max_{\left|k\right|\leq C\left(2^l\log l\right)^{\frac{1}{4}}}\frac{1}{2^{l-1}c_n}\left|Q_n\left(t_p+d_nk\right)-Q_n\left(t_p\right)\right|\right)^2\\
\leq \frac{1}{2^{2(l-1)}c_{2^l}^2}E\left(\max_{\left|k\right|\leq C\left(2^l\log l\right)^{\frac{1}{4}}}\sum_{d=1}^{l}\max_{i=1,\ldots,2^{l-d}}\left(Q_{2^{(l-1)}+i2^{(d-1)}}\left(t_p+d_nk\right)-Q_{2^{(l-1)}+i2^{(d-1)}}\left(t_p\right)\right)\right)^2\displaybreak[0]\\
\leq \frac{1}{2^{2(l-1)}c_{2^l}^2}\sum_{\left|k\right|\leq C\left(2^l\log l\right)^{\frac{1}{4}}}l\sum_{d=1}^{l}\sum_{i=1}^{2^{l-d}}E\left(Q_{2^{(l-1)}+i2^{(d-1)}}\left(t_p+d_nk\right)-Q_{2^{(l-1)}+i2^{(d-1)}}\left(t_p\right)\right)^2\displaybreak[0]\\
\shoveleft\leq\frac{1}{2^{2(l-1)}c_{2^l}^2}\sum_{\left|k\right|\leq C\left(2^l\log l\right)^{\frac{1}{4}}}l^2\sum_{i_1,i_2,i_3,i_4=1}^{2^l}\\
\left|E\left|\left(h_2\left(X_{i_1},X_{i_2},t_p+d_nk\right)-h_2\left(X_{i_1},X_{i_2},t_p\right)\right)\left(h_2\left(X_{i_3},X_{i_4},t_p+d_nk\right)-h_2\left(X_{i_3},X_{i_4},t_p\right)\right)\right]\right|,
\end{multline*}
where we used the triangular inequality in the last step. By means of Lemma \ref{lem7} and the same arguments as in the proof of Lemma 2 of Yoshihara \cite{yosh}, we arrive at
\begin{multline*}
 E\left(\max_{2^{l-1}\leq n<2^l}\max_{\left|k\right|\leq C\left(2^l\log l\right)^{\frac{1}{4}}}\frac{1}{2^{l-1}c_n}\left|Q_n\left(t_p+d_nk\right)-Q_n\left(t_p\right)\right|\right)^2\\
\leq \frac{C}{2^{4l}c_{2^l}^2}\left(\frac{2^l}{\log l}\right)^{\frac{1}{4}}l^2 2^{2l}\sum_{i=1}^{2^l}i\alpha^{\frac{1}{2}}\left(i\right)\leq \frac{C2^{l(\frac{3}{2}+\frac{1}{4}\gamma)}}{2^{4l}l^{\frac{3}{2}}(\log l)^{\frac{5}{4}}}l^2 2^{l(4-\frac{\beta}{2})}=C\frac{2^{l(\frac{3}{2}+\frac{1}{4}\gamma-\frac{1}{2}\beta)}l^{\frac{1}{2}}}{(\log l)^{\frac{5}{4}}}.
\end{multline*}
As $\beta>\frac{7}{2}$, we have that $\frac{3}{2}+\frac{1}{4}\gamma-\frac{1}{2}\beta=\frac{-2\beta^2+7\beta-2}{4\beta}<0$ and thus the second moments are summable. Line (\ref{line31}) follows by the Chebyshev inequality and the Borel-Cantelli lemma, so Line (\ref{line16}) is proved.

Under Condition 2. (functionals of absolutely regular sequences), we have by Lemma \ref{lem8} and $\sum_{i=1}^{\infty}i\beta(i)<\infty$, $\sum_{i=1}^{\infty}iA_i<\infty$
\begin{multline*}
E\left(\max_{2^{l-1}\leq n<2^l}\max_{\left|k\right|\leq C\left(2^l\log l\right)^{\frac{1}{4}}}\frac{1}{2^{l-1}c_n}\left|Q_n\left(t_p+d_nk\right)-Q_n\left(t_p\right)\right|\right)^2\\
\leq \frac{C}{2^{4l}c_n^2}\left(\frac{2^l}{\log l}\right)^{\frac{1}{4}}l^2 2^{2l}\sum_{i=1}^{2^l}i\left(\beta(\frac{i}{3})+A_{\frac{i}{3}}\right)\leq \frac{C2^{l(\frac{3}{2}+\frac{1}{4}\gamma)}}{2^{4l}l^{\frac{3}{2}}(\log l)^{\frac{5}{4}}}l^2 2^{2l}=\frac{Cl^{\frac{1}{2}}}{2^{l(\frac{1}{2}-\frac{1}{4}\gamma)}(\log l)^{\frac{5}{4}}}.
\end{multline*}
Since $\gamma\in(0,1)$, we have that $\frac{1}{2}-\frac{1}{4}\gamma>0$ and the second moments are summable. Line (\ref{line31}) follows by the Chebyshev inequality and the Borel-Cantelli lemma, so Line (\ref{line16}) is proved.

To prove Line (\ref{line17}), let w.l.o.g. $u\left(t_p\right)=1$, otherwise replacing $h(x,y,t)$ by $h\left(x,y,\frac{t}{u(t_p)}\right)$. We represent $R'_n$ as $Z'_n\left(U\left(t_p\right)-U_n\left(t_p\right)\right)$ with
\begin{equation*}
 Z'_n\left(x\right):=\left(U_n\left(\cdot+t_p\right)-U_n\left(t_p\right)\right)^{-1}\left(x\right)-x=U_n^{-1}\left(x+U_n\left(t_p\right)\right)-x-t_p.
\end{equation*}
By Proposition \ref{theo4}
\begin{equation*}
 \limsup_{n\rightarrow\infty}\pm\sqrt{\frac{n}{\log\log n}}\left(U_n\left(t_p\right)-U\left(t_p\right)\right)=C.
\end{equation*}
By Line (\ref{line16}) and Condition (\ref{line15})
\begin{multline*}
\sup_{\left|x\right|\leq C\sqrt{\frac{\log\log n}{n}}}\left|U_n\left(x+t_p\right)-U_n\left(t_p\right)-x\right|\\
\leq \sup_{\left|x\right|\leq C\sqrt{\frac{\log\log n}{n}}}\left|U_n\left(x+t_p\right)-U\left(x+t_p\right)-U_n\left(t_p\right)+U\left(t_p\right)\right|\\
+\sup_{\left|x\right|\leq C\sqrt{\frac{\log\log n}{n}}}\left|U\left(x+t_p\right)-U\left(t_p\right)-x\right|=o\left(c_n\right).
\end{multline*}
Then by Theorem 1 of Vervaat \cite{verv}
\begin{equation*}
\left|R_n'\right|\leq\sup_{\left|x\right|\leq C\sqrt{\frac{\log\log n}{n}}}\left|Z'_n\left(x\right)\right|=o\left(c_n\right),
\end{equation*}
so Line (\ref{line17}) is proved.
\end{proof}

\section*{Acknowledgements}
I want to thank Herold Dehling, who proposed studying this topic and discussed it with me many times.

\small{

}
\end{document}